\documentclass[a4paper, reqno, 11pt]{amsart}
\usepackage[english]{babel}
\usepackage{amsmath}
\usepackage{amssymb}
\usepackage{amsthm}
\usepackage{enumerate}
\usepackage{ifthen}
\usepackage{bbm}
\usepackage{color}
\provideboolean{shownotes} 
\setboolean{shownotes}{true}
\usepackage{hyperref}
\usepackage{geometry}
\newcommand{\margnote}[1]{
\ifthenelse{\boolean{shownotes}}%
{\marginpar{\raggedright\tiny\texttt{#1}}}%
{}%
}
\newcommand{\hole}[1]{
\ifthenelse{\boolean{shownotes}}%
{\begin{center} \fbox{ \rule {.25cm}{0cm}
\rule[-.1cm]{0cm}{.4cm} \parbox{.85\textwidth}{\begin{center}
\texttt{#1}\end{center}} \rule {.25cm}{0cm}}\end{center}}
{}
}
\newtheorem{theorem}{Theorem}[section]
\newtheorem{proposition}[theorem]{Proposition}
\newtheorem{lemma}[theorem]{Lemma}

\newtheorem{definition}[theorem]{Definition}

\theoremstyle{remark}
\newtheorem{remark}[theorem]{Remark}
 
\newcommand{\R}{\mathbb{R}} 
\newcommand{\tore}{\mathbb{T}^3}
\newcommand{\T}{\mathbb{T}^3} 
\newcommand{\N}{\mathbb{N}} 
\newcommand{\Z}{\mathbb{Z}}

\newcommand{\weakto}{\rightharpoonup}

\newcommand{\ub}{u^{\be}} 
\newcommand{\vn}{v^{n}}
 
\newcommand{\pn}{p^{n}} 

\newcommand{\pb}{p^{\be}} 
\newcommand{\be}{\alpha}

\newcommand{\un}{u^{n}} 

\newcommand{\um}{u^{\alpha,m}_{n}}
\newcommand{\umu}{u^{\alpha,m-1}_{n}}
\newcommand{\ui}{u^{\alpha,i}_{n}}
\newcommand{\uiu}{u^{\alpha,i-1}_{n}}
\newcommand{\uM}{u^{\alpha,M}_{n}}
\newcommand{\vM}{v^{\alpha,M}_{n}}
\numberwithin{equation}{section}
\subjclass[2010]{Primary: 35Q30, Secondary: 35A35, 76M20.}
\keywords{Navier-Stokes Equations, Entropy Solutions, Large Eddy Simulation,
  Navier-Stokes-Voigt model, space-time discretization}
\begin{document}
\title[On the convergence of a fully discrete scheme of LES type]{On the convergence of a
  fully discrete scheme of LES type to physically relevant solutions of the incompressible
  Navier-Stokes}
\author[L.C. Berselli]{Luigi C. Berselli} 
\address[L. C. Berselli]{ Dipartimento di Matematica
  \\
Universit\`a di Pisa\\   I-56127, Pisa, Italy}
\email[]{\href{luigi.carlo.berselli@unipi.it}{luigi.carlo.berselli@unipi.it}}
\author[S. Spirito]{Stefano Spirito} 
\address[S. Spirito]{DISIM - Dipartimento di Ingegneria e Scienze dell'Informazione e
  Matematica
  \\
 Universit\`a degli Studi dell'Aquila\\
 I-67100,  L'Aquila, Italy.}
\email[]{\href{stefano.spirito@univaq.it}{stefano.spirito@univaq.it}}
\begin{abstract}
  Obtaining reliable numerical simulations of turbulent fluids is a challenging problem in
  computational fluid mechanics. The Large Eddy Simulations (LES) models are efficient
  tools to approximate turbulent fluids and an important step in the validation of these
  models is the ability to reproduce relevant properties of the flow. In this paper we
  consider a fully discrete approximation of the Navier-Stokes-Voigt model by an implicit
  Euler algorithm (with respect to the time variable) and a Fourier-Galerkin method (in
  the space variables). We prove the convergence to weak solutions of the incompressible
  Navier-Stokes equations satisfying the natural local entropy condition, hence selecting
  the so-called \textit{physically relevant} solutions.
\end{abstract}
\maketitle
%
%
%
\section{Introduction}
We consider the incompressible Navier-Stokes Equations (NSE) with periodic boundary
conditions
\begin{equation}
  \label{eq:nse}
  \begin{aligned}
    \partial_t u-\Delta u+(u\cdot\nabla)\,u+\nabla p&=0\qquad&\textrm{ in }
    (0,T)\times\tore,
    \\
    \nabla\cdot u&=0\qquad&\textrm{ in } (0,T)\times\tore,
    \\
    u|_{t=0}&=u_0\qquad&\textrm{ on }\tore,
  \end{aligned}
\end{equation}
where $T>0$ is arbitrary and $\tore$ the three dimensional flat torus.  Here the velocity
field $u\in\R^3$ as well as the pressure $p$ are space periodic and with zero mean
value. Even if turbulent phenomena arise for large values of the Reynolds number, we set
here the viscosity equal to one and the external force equal to zero, since these
assumptions do not affect the main result. \par
%
Obtaining an accurate prediction (of averaged quantities) of turbulent fluids is a
central difficulty in computational fluid mechanics and we recall that direct numerical
simulations have --at present-- an unaffordable computational costs to perform this
task. The most promising tools to perform accurate simulations of turbulent fluids are
given by the Large Eddy Simulations (LES) models. LES models are based on the idea that
in many practical situations it is enough to simulate the mean characteristics of the
flow by averaging/filtering the equations.  A very popular LES model is given the
Navier-Stokes-Voigt equations, whose Cauchy problem reads as follows:
\begin{equation}
  \label{eq:nsve}
  \begin{aligned}
    \partial_{t}(\ub_t-\be^2\Delta\ub)-\Delta
    \ub+(\ub\cdot\nabla)\,\ub+\nabla\pb&=0\quad&\text{in
    }(0,T)\times\tore,
    \\
    \nabla\cdot\ub&=0\quad&\text{in }(0,T)\times\tore,
    \\ 
    \ub|_{t=0}&=u_0\quad&\text{on }\quad\tore.
  \end{aligned}
\end{equation}
Here, the parameter $\be>0$ has the dimension of a length and roughly speaking the scales
smaller than $\be$ are truncated. It is also well-known that for system~\eqref{eq:nsve}
one can prove global existence and uniqueness of solutions.  We refer
to~\cite{CLT2006,LaTi2010} for the analysis of the Cauchy problem~\eqref{eq:nsve} and for
the interpretation of the results. In particular the regularization introduced by the
operator $-\partial_{t}\Delta$ is of hyperbolic type (not an extra dissipation as in eddy
viscosity models) and the system is of pseudo-parabolic type. To assess the model from the
mathematical point of view one important question is to show that the solutions, in the
limit as $\alpha\to0$ produce weak solutions of the Navier-Stokes equations, which satisfy
the local energy inequality
\begin{equation}
  \label{eq:leiintro}
  \partial_t\left(\frac{|u|^2}{2}\right)+\nabla\cdot\left(\left(\frac{|u|^2}{2}+
      p\right)u\right)-\Delta \left(\frac{|u|^2}{2}\right)+|\nabla u|^2\leq 0,
\end{equation}
in the sense of distributions over $(0,T)\times \tore$.  

We recall that starting with the results on global existence of weak solutions for the NSE
by Leray~\cite{Le} and Hopf~\cite{Ho} a still unsolved problem is that of uniqueness and
regularity of these solutions. Moreover, among weak solutions those satisfying the local
energy inequality~\eqref{eq:leiintro} are of particular importance because for them holds
true the celebrated partial regularity theorem of
Caffarelli-Kohn-Nirenberg~\cite{CKN1982}. Finally, we notice that the
inequality~\eqref{eq:leiintro} is a natural request that solutions constructed by
numerical methods should satisfy, see Guermond~\cite{Gue2008,GOP2004}.  A weak solution
of~\eqref{eq:nse} satisfying~\eqref{eq:leiintro} is known in literature as suitable weak
solutions. The first existence result of suitable weak solutions is due to
Caffarelli-Kohn-Nirenberg~\cite{CKN1982}. Then, the convergence to suitable weak solutions
has been proved for different methods, see~\cite{Bei1985a, Bei1985b,BS2017b,DS2011}, but
the approximation methods are of all of ``{\em infinite dimensional type}'', that is
obtained by approximating the NSE~\eqref{eq:nse} by another system of partial differential
equations, and few results are available when the approximation methods are finite
dimensional as in numerical methods. In~\cite{Gue2006, Gue2007} Guermond proved the
convergence to a suitable weak solution for numerical solutions obtained by using some
finite element Galerkin methods (only with respect to the space variables), while some
conditional results on Fourier based Galerkin methods on the torus are proved
in~\cite{BCI2007}. In particular, the convergence to a suitable weak solution of the
standard Fourier-Galerkin method is still an interesting open problem and the
space-periodic setting and the use of Fourier series expansion is not an assumption to
simplify the technicalities. From the numerical point of view another important issue is the
time discretization. In~\cite{BS2016} it is proved that solutions of periodic
Navier-Stokes equations constructed by the standard implicit Euler algorithm are
suitable. The result has been later extended to a general domain in assuming at the
boundary slip vorticity based conditions, which are important in the vanishing viscosity
problem~\cite{BS2012, BS2014}. The case of Dirichlet boundary conditions is still open.

The aim of this paper is to perform a space-time full discretization of~\eqref{eq:nsve}
and to prove the convergence, varying the parameters of the numerical discretization and
as $\be\to0$, to weak solutions of Navier-Stokes equations satisfying the local energy
inequality

In order to discretize in time~\eqref{eq:nsve} we use the implicit Euler algorithm, while
in space we use the spectral Galerkin methods, based on Fourier series expansion
\begin{equation}
  \label{eq:approx}
  d_t(\um-\alpha^2\Delta\um)-\Delta\um+P_{n}((\um\cdot\nabla)\,\um)=0, 
\end{equation}
where $d_{t}$ denotes the finite difference operator and where $P_{n}$ is the projection
over the space of Fourier modes smaller of equal than $n$, see Section~\ref{sec:disc} for
the precise formulations of the discretization. Here we only point out that the output of
this Euler-Fourier-Galerkin type of approximation is a triple $(u^{\be, M}_n, v^{\be,M}_n,
p^{\be, M}_n)$, where $M\in\N$ is the parameter defining the time-step $\kappa=T/M$.
The main result of this paper is the following theorem. See Section~\ref{sec:pre} for the
notations concerning the spaces.
 \begin{theorem}
\label{teo:main}
Let $u_0\in H^{2}_{0,\sigma}$ and $\{(u^{\be, M}_n, v^{\be,M}_n,p^{\be,
  M}_n)\}_{(n,\be,M)}$ be a sequence of solutions of the approximating
Euler-Fourier-Galerkin scheme of~\eqref{eq:approx}. Let $\{M_n\}_{n}\subset\N$ be any
monotone sequence converging to infinity and let $\be_n\subset(0,1)$ be any monotone
sequence converging to zero and such that
\begin{equation}
  \label{eq:ass}
  \lim_{n\to\infty}n\,\be_n^3=0.
\end{equation}
Then, there exists 
\begin{equation*}
(u,p)\in L^\infty(0,T;L_{0,\sigma}^2)\cap
L^2(0,T;H^{1}_{0,\sigma})\times L^{5/3}((0,T)\times\tore),
\end{equation*}
such that, up to a subsequence not relabelled, the following convergence holds true as
$n\to\infty$:
\begin{equation*}
\begin{aligned}
  &v^{\be_n,M_n}_n\rightarrow u\textrm{ strongly in }L^{2}((0,T)\times\tore),
  \\
  &u^{\be_n, M_n}_n\rightarrow u\textrm{ strongly in }L^{2}((0,T)\times\tore),
  \\
  &\nabla u^{\be_n, M_n}_n\weakto \nabla u\textrm{ weakly in }L^{2}((0,T)\times\tore),
  \\
  &p^{\be_n, M_n}_n\weakto p\textrm{ weakly in } L^{5/3}((0,T)\times\tore).
\end{aligned}
\end{equation*}
Moreover, $(u,p)$ is a suitable weak solution of~\eqref{eq:nse} in the sense of
Definition~\ref{def:suit}.
\end{theorem}
\begin{remark}
  The assumption on the initial datum can be relaxed, by an appropriate regularization. We
  do not state and prove Theorem~\ref{teo:main} under this more general hypothesis in
  order to avoid further technicalities.
\end{remark}
\begin{remark}
  We note that, while the sequence $\{\be_n\}_n$ is related to $n$ by~\eqref{eq:ass}, the
  sequence $\{M_n\}_n$ is arbitrary. This means that there is no need to link the time and
  the space approximation in order to have convergence of the
  scheme. Theorem~\ref{teo:main} may be equivalently stated in term of a double sequence
  $\{(u^{\be_n, M}_n, v^{\be_n,M}_n, p^{\be_n, M}_n\}_{(M,n)}$ and the convergences hold
  as $(M,n)\to\infty$.
\end{remark} 
The convergence of Fourier-Galerkin method of~\eqref{eq:nsve} to a suitable weak solutions
of~\eqref{eq:nse}, without the time discretization, but with $\be_n$
satisfying~\eqref{eq:ass} has been proved as one of the results in~\cite{BS2017b}. Here
new difficulties arise from the non trivial combinations of the time discretization and
the proof of certain discrete {\em a priori} estimates which are counterpart of those
obtained in~\cite{BS2017b}.

The problem of the convergence of numerical schemes to solutions satisfying local
energy-type balance is present also in several other equations in fluid mechanics. Among
them we want to cite the case of the two dimensional Euler equations with vorticity in
$L^p$. In this case, satisfying the local energy balance is almost equivalent to solve the
vorticity equations in the renormalized sense and the additional information obtained is
that the solution obtained is Lagrangian, we refer to~\cite{CS,CNSS} for
further details.

\smallskip

\textbf{Plan of the paper.}  In Section~\ref{sec:pre} we fix the notation that we use in
the paper, we recall the main definitions regarding the NSE~\eqref{eq:nse}, and the tools
used. In Section~\ref{sec:disc} we introduce and describe in details the space-time
discretization we consider. In Section~\ref{sec:apriori} we prove the main {\em a priori} estimates
needed to study the convergence and finally in Section~\ref{sec:5} we prove
Theorem~\ref{teo:main}.
\section{Preliminaries}
\label{sec:pre}
In this section we give details on the functional setting and then we recall the main
definitions concerning weak solutions of incompressible Navier-Stokes equations.
\subsection{Notations}
We introduce the notations typical of space-periodic problems. The three dimensional torus
is defined by $\tore := \R^3 / 2\pi\Z^3$. We denote with
$C^{\infty}_{c}(I;C^{\infty}(\tore))$ the space of smooth functions or vectors which are
compactly supported on the interval $I\subset\R$ and $2\pi$-periodic with respect to the
space variables. In the sequel we shall use the customary Lebesgue spaces $L^p(\tore)$ and
Sobolev spaces $W^{k,p}(\tore)$ and we shall denote their norms by $\|\cdot\|_p$ and
$\|\cdot\|_{W^{k,p}}$. Moreover, in the case $p=2$ we use the notation
$H^s(\tore):=W^{s,2}(\tore)$ and, for simplicity, we shall not distinguish between scalar
and vector valued functions.
Finally, we use $\left(\cdot,\cdot\right)$ to denote the $L^{2}(\tore)$ paring. Since we
are working in the periodic setting we denote by the subscript ``$_{0}$'' the subspaces of
zero average vectors of $L^{2}(\tore)$ and $H^{s}(\tore)$, for any
exponent $s\geq0$.
The divergence-free constraint is also directly included in the function spaces in the
analysis of the NSE, and as usual we define
 \begin{equation*}
\begin{aligned}
  &L^{2}_{0,\sigma}:= \left\{w :\tore \rightarrow \R^3, \, \, w \in
    L^{2}(\tore),\quad\nabla\cdot w = 0 \quad\int_{\tore}w\,dx = {0} \right\},
  \\
  &H^{s}_{0,\sigma}:= \left\{w :\tore \rightarrow \R^3, \, \, w \in
    H^{s}(\tore),\quad\nabla\cdot w = 0 \quad\int_{\tore}w\,dx = {0} \right\},
 \end{aligned}
\end{equation*}
and we recall that the divergence condition can be easily defined in terms of the Fourier
coefficients. For any $s>0$ we denote by $H^{-s}:=(H^{s}_{0,\sigma})'$.\par
Finally, the space $L^p(0,T;X)$, where $X$ is a Banach space, is the classical Bochner
spaces endowed with its natural norm denoted by $\|\cdot\|_{L^{p}(X)}$.
\subsection{Leray-Hopf and Suitable Weak Solutions}
We start by recalling the definition of weak solution of the initial value
problem~\eqref{eq:nse}, as introduced by Leray and Hopf.  
\begin{definition}[Leray-Hopf Weak Solutions] 
  \label{def:lws}
  The vector field $u\in L^\infty(0,T;L_{0,\sigma}^2)\cap L^2(0,T;H^{1}_{0,\sigma})$ is a
  Leray-Hopf weak solution of~\eqref{eq:nse} if:
  \begin{enumerate}
  \item $u$ satisfies the following identity
    \begin{equation*}
      \int_0^{T} \left(u,\partial_t\varphi\right)-\left(\nabla
        u,\nabla\varphi\right)-\left((u\cdot\nabla)\,u,\varphi\right)dt+\left(u_0,\varphi(0)\right)\,dt=0, 
\end{equation*}
for all smooth, periodic and divergence-free functions $\varphi\in
C^{\infty}_c([0,T);C^{\infty}(\tore))$ with zero mean value over $\tore$.
\item The following energy inequality holds true:
\begin{equation*}
  \frac12\|u(t)\|_{2}^2+\int_0^t\|\nabla u(s)\|_{2}^2\,ds\leq\frac12 \|u_0\|_{2}^2\quad\textrm{
    for all }t\in[0,T]. 
\end{equation*}
\end{enumerate}
\end{definition}
We remark that $u$ attains the initial datum in the strong sense, namely
\begin{equation*}
  \lim_{t\to 0^{+}}\|u(t)-u_0\|_{2}=0.
\end{equation*}
Suitable weak solutions are a particular subclass of Leray-Hopf weak solutions. They were
introduced by Scheffer in~\cite{Sche1977} and Caffarelli-Kohn-Nirenberg
in~\cite{CKN1982}. The definition in the periodic setting is the following.
\begin{definition}[Suitable Weak Solutions]
\label{def:suit}
A pair $(u,p)$ is a Suitable Weak Solution to the Navier-Stokes equation~\eqref{eq:nse} if
$u$ is a Leray-Hopf weak solution, if $p\in L^{\frac{5}{3}}((0,T)\times\tore)$, and if the
local energy balance holds true 
\begin{equation}
  \label{eq:lei}
  \int_0^T\int_{\tore}|\nabla u|^2\phi\,dxdt\leq
  \int_{0}^{T}\int_{\tore}\left[\frac{|u|^2}{2}\left(\partial_t\phi+\Delta\phi\right)
+\left(\frac{|u|^2}{2}+p\right)u\cdot\nabla \phi\right]\,dxdt.
\end{equation}
%
for all $\phi\in C^\infty_0(0,T;C^\infty(\tore))$ such that $\phi\geq0$.
\end{definition}
\section{Time-Discrete Fourier-Galerkin Methods}
\label{sec:disc}
In this section we introduce the space-time full discretization of the Navier-Stokes-Voigt
equations~\eqref{eq:nsve} we are going to analyze.  Let $P$ denote the Leray projector of
$L^2_{0}(\T)$ onto $L^2_{0,\sigma}$, which  explicitly reads in the orthogonal
Hilbert basis of complex exponentials as follows:
\begin{equation*}
  P:\ g(x)=\sum_{k\in\mathbb{Z}^3\backslash\{0\}}\hat{g}_k\, \text{e}^{i  k\cdot x}\
  \mapsto\
  Pg(x)=\sum_{k\in\mathbb{Z}^3\backslash\{0\}}\left[\hat{g}_k-\frac{(\hat{g}_k\cdot k) 
      k}{|k|^2}\right]\, \text{e}^{i k\cdot x}.  
\end{equation*}
Then, for any $n\in\N$, we denote by $P_n$ the projector of $L_0^2(\T)$ on the
finite-dimensional sub-space $V_n:=P_n (L^2_{0,\sigma})$, given by the
following expression
\begin{equation*}
  P_n:g(x)=\sum_{k\in\mathbb{Z}^3\backslash\{0\}}\hat{g}_k \,\text{e}^{ik\cdot x}\
  \mapsto\ 
  P_n g(x)=\sum_{0<|k|\leq n}\left[\hat{g}_k-\frac{(\hat{g}_k\cdot k) k}{|k|^2}\right]\,
  \text{e}^{i k\cdot x}.  
\end{equation*}
The (space) approximate Fourier-Galerkin method to~\eqref{eq:nsve} is given by the following
system%
\begin{equation}
  \label{eq:a1}
  \begin{aligned}
    \partial_t (u_n^{\be}-\be^2\Delta u_n^{\be})- \Delta
    u_n^{\be}+P_n((u_n^{\be}\cdot\nabla)\,u_n^{\be})&=0&\hspace{-.3cm}\text{ in
    }(0,T)\times\T,
    \\
    u_n^{\be}|_{t=0}&=P_nu_0 &\text{ in }\T,
  \end{aligned}
\end{equation}
where 
\begin{equation}\label{eq:a2}
  u_n^{\be}(t,x)=\sum_{0<|k|\leq n}\hat{u}_{n,k}^{\be}(t)\,\text{e}^{
    ik\cdot x},\qquad\text{with}\quad k\cdot \hat{u}^{\be}_{n,k}=0.
\end{equation}
We note that the divergence-free condition is encoded in~\eqref{eq:a2} and~\eqref{eq:a1}
is a (finite dimensional) system of ODEs in the unknowns $\hat{u}_{n,k}^{\be}(t)$.\par
Next, we proceed by performing the time discretization of~\eqref{eq:a1} by finite
differences in time. Let $M\in\N$ and $\kappa=T/M$. We consider the net
$I^M=\{t_m\}_{m=0}^{M}$ with $t_0=0$ and $t_m=m\kappa$ and discretize~\eqref{eq:a1} by
using the implicit Euler algorithm: Set $u^{\be,0}_{n}=P_nu_0$. For any $m=1,...,M$, given
$\umu\in V_n$ find $\um\in V_n$ by solving
\begin{equation}
  \label{eq:a3}
  d_t(\um-\alpha^2\Delta\um)-\Delta\um+P_{n}((\um\cdot\nabla)\,\um)=0.  
\end{equation}
where
\begin{equation}
  \label{eq:a4}
\um(x)=\sum_{0<|k|\leq n}\hat{u}_{n,k}^{\,\be,m}\,\text{e}^{ik\cdot
  x},\qquad\text{with}\quad k\cdot \hat{u}_{n,k}^{\,\be,m}=0, 
\end{equation}
and 
\begin{equation*}
d_t\um:=\frac{u^{\alpha,m}_{n}-u^{\alpha,m-1}_{n}}{\kappa}.
\end{equation*}
We point out that again the divergence-free condition is enforced by~\eqref{eq:a4} and
now, for each $m=1,\dots,M$, the system~\eqref{eq:a3} is a finite dimensional nonlinear
(algebraic) system, in the unknowns $\hat{u}_{n,k}^{\,\be,m}\in \R$. 

Finally, since we are considering the periodic setting we can define the associated
approximation for the pressure by solving the Poisson problem
\begin{equation}
  \label{eq:poiss}
  -\Delta p^{\alpha,m}_{n}=\nabla\cdot\big(\nabla\cdot(\um\otimes\um)\big)\qquad
  m=1,\dots,M, 
\end{equation}
with periodic boundary conditions and zero mean value on $p^{\alpha,m}_{n}$. Moreover, in
order to prove the convergence to a suitable weak solution, it will turn out to be
convenient to (re)formulate the equations~\eqref{eq:a3} as follows
\begin{equation}
  \label{eq:a6}
\begin{aligned}
  &d_t(\um-\alpha^2\Delta\um)-\Delta\um+(\um\cdot\nabla)\,\um
  -Q_{n}((\um\cdot\nabla)\,\um)+\nabla p^{\alpha,m}_{n}=0.
\end{aligned}
\end{equation}
where the operator $Q_{n}$ is defined by $Q_n:=P-P_n$.\par
As usual in the study of finite difference numerical schemes, we can now rephrase the
problem~\eqref{eq:a3} on $(0,T)\times\tore$, by introducing the following time dependent
functions
 \begin{equation}
 \label{eq:vm}
 \begin{aligned}
   &u_{n}^{\be,M}(t)=\left\{
     \begin{aligned}
       &\um & \text{for
       } t\in[t_{m-1},t_m),
       \\
       &u_{n}^{\be,M}\qquad &\text{for } t=t_M,
\end{aligned}
\right.
\\
   &v_{n}^{\be,M}(t)=\left\{
     \begin{aligned}
       &u_{n}^{\be,m-1}+\frac{t-t_{m-1}}{\kappa}(\um-u_{n}^{\be,m-1})\ & \text{for
       } t\in[t_{m-1},t_m),
       \\
       &u^{\be,M}_n\qquad &\text{for } t=t_M,
     \end{aligned}
   \right.\\
 &p_{n}^{\be,M}(t)=\left\{
  \begin{aligned}
    &p_{n}^{\alpha,m} & \text{for
    } t\in[t_{m-1},t_m),
    \\
    &p_{n}^{\alpha,M}\qquad &\text{for } t=t_M.
\end{aligned}
\right.
\end{aligned}
\end{equation}
Then, equations~\eqref{eq:a3} read as follows  
\begin{equation}
  \label{eq:a7}
  \partial_t(\vM-\be^2\Delta\vM)-\Delta\uM+P_n((\uM\cdot\nabla)\,\uM)=0,
\end{equation}
and, accordingly, Eq.~\eqref{eq:a6} on $(0,T)\times\tore$ becomes
\begin{equation}
  \label{eq:a8}
  \begin{aligned}
    &\partial_t(\vM-\alpha^2\Delta\vM)-\Delta\uM+(\uM\cdot\nabla)\,\uM
    -Q_{n}((\uM\cdot\nabla)\,\uM)+\nabla p^{\alpha,M}_{n}=0. 
  \end{aligned}
\end{equation}
We stress that in order to prove the convergence to a suitable weak solution, it is
crucial to prove that the term involving $Q_n$ goes to zero as $n\to\infty$. To this end
we recall the following lemma, which is proved as one of the main steps
in~\cite[Lemma~4.4]{BCI2007}.
\begin{lemma}
  \label{lem:2.4}
  Let be given $\phi\in C^{\infty}((0,T)\times\T)$ and let $u^n$ be defined as
  \begin{equation*}
    u^n(t,x):=\sum_{0<|k|\leq n}\widehat{U}_k^{n}(t)\,\text{e}^{ik\cdot x}.
  \end{equation*}
  Then, there exists a constant $c$, depending only on $\phi$ (but independent of
  $n\in\N$), such that
\begin{equation*}
  \|Q_n(u^n(t)\phi(t))\|_{\infty}^2\leq
  c\left(n^2\sum_{|k|\geq\frac{n}{2}}|\widehat{U}^{n}_k(t)|^2+
    \frac{1}{n}\sum_{k\in\mathbb{Z}^3}|\widehat{U}^{n}_k(t)|^2\right).   
  \end{equation*}
\end{lemma}
\section{A Priori Estimates}
\label{sec:apriori}
In this section we prove the {\em a priori} estimates needed to prove the convergence
to~\eqref{eq:nse}. We start with the following basic discrete energy inequality.
\begin{lemma}
  \label{lem:discene}
  Let be given $u_0\in H^{2}_{0,\sigma}$. Let $\um$ be a solution of~\eqref{eq:a3}. Then
  following discrete energy equality holds true for all $M\in\N$ and $m=1,..,M$
\begin{equation}
\label{eq:a}
\begin{aligned}
  \|\um\|_{2}^2+&\sum_{i=1}^{m}\|\ui-\uiu\|_{2}^2+2\kappa\sum_{i=1}^{m}\|\nabla\ui\|_{2}^2
  \\
  &+\alpha^2\|\nabla\um\|_{2}^2+\alpha^2\sum_{i=1}^{m}{\|\nabla\ui-\nabla\uiu\|_{2}^2}
  =\|u_0\|_{2}^2+\alpha^2\|\nabla u_0\|_{2}^2.
\end{aligned}
\end{equation}
\end{lemma}
\begin{proof}
  Fix $M\in \N$ and $m=1,...,M$. Consider the equations~\eqref{eq:a1} for $i=1,...,m$ and
  multiply~\eqref{eq:a1} by $\ui$.  Then, after integration by parts over $\tore$ we get
\begin{equation*}
  \left(\frac{\ui-\uiu}{\kappa},\ui\right)+\alpha^2\left(\frac{\nabla\ui-
      \nabla\uiu}{\kappa},\nabla\ui\right)+\|\nabla\ui\|_{2}^2=0,
\end{equation*}
where we used that fact that since $\ui\in V_n$ then
\begin{equation*}
  (P_n((\ui\cdot\nabla)\,\ui),\ui)=0.
\end{equation*}
By using the elementary equality 
\begin{equation}
\label{eq:elem}
(a, b-a)=\frac{|a|^2}{2}-\frac{|b|^2}{2}+\frac{|a-b|^{2}}{2},
\end{equation}
the terms involving the discrete derivative become the following:
\begin{align*}
  & (\ui-\uiu,\ui)=\frac{1}{2}(\|\ui\|_2^2-\|\uiu\|_2^2)+\frac{1}{2}\|\ui-\uiu\|_2^2, 
  \\
  & (\nabla\ui-\nabla\uiu,\nabla\ui)=\frac{1}{2}(\|\nabla\ui\|_2^2-\|\nabla\uiu\|_2^2)+ 
  \frac{1}{2}\|\nabla\ui-\nabla\uiu\|_2^2.
\end{align*}
Finally, by summing up over $i=1,...,m$ we get~\eqref{eq:a}.
\end{proof}
The next lemma regards two weighted estimates on higher derivatives of solutions
of~\eqref{eq:a3} and they will be useful when proving the convergence to a suitable weak
solution. The results in the following lemma are a discrete counterpart of those proved
in~\cite{BS2017b}.
\begin{lemma}
  \label{lem:2}
  Let $u_0\in H^2_{0,\sigma}$ and $\be\leq1$. Let $M\in\N$ and $m=1,...,M$. Let $\um$ be a
  solution of~\eqref{eq:a3}. Then, there exists $c>0$, independent of $\be>0$, of $M\in\N$
  and of $n\in\N$, such that
  \begin{align}
    &\be^{3} \kappa\sum_{m=1}^{M}\|d_t\um\|_2^2\leq c,\label{eq:tdw}
    \\
    & \be^{6}\kappa\sum_{m=1}^M\|\Delta \um\|_2^2\leq c.\label{eq:2dw}
  \end{align}
\end{lemma} 
\begin{proof}
  Let $M\in\N$ and $m=1,...,M$. We multiply~\eqref{eq:a3} by $\be^3d_t\um$. After
  integrating by parts over $\tore$ we get 
  \begin{equation*}
    \begin{aligned}
      \alpha^3(d_t\nabla\um,\nabla\um)&+\be^3\|d_t\um\|_2^2+{\be^5}\|d_t\nabla\um\|_2^2
      +\be^3(P_n((\um\cdot\nabla)\,\um,d_t\um)=0.
    \end{aligned} 
  \end{equation*}
  By using~\eqref{eq:elem} we then get
  \begin{equation}
    \label{eq:p1}
    \begin{aligned}
      &\frac{\be^3}{2}(\|\nabla\um\|_2^2-\|\nabla\umu\|_2^2)+
      \frac{\be^3}{2}\|\nabla\um-\nabla\umu\|_2^2
      \\
      &\qquad+\be^3\kappa\|d_t\um\|_2^2+{\be^5\kappa}\|d_t\nabla\um\|_2^2\leq
      \be^3\kappa|((\um\cdot\nabla)\,\um,d_t\um)|,
    \end{aligned}
  \end{equation}
  where we used the fact that $d_t\um\in V_n$. By using H\"older and Gagliardo-Nirenberg
  inequalities we estimate the right hand side as follows
  \begin{align*}
  \be^3\kappa|((\um\cdot\nabla)\,\um,d_t\um)|
  &\leq\alpha^3\kappa\|\um\|_{4}\|\nabla \um\|_2 \|d_t\um\|_{4}
  \\
  &\leq c\alpha^3\kappa\|\um\|_2^{\frac{1}{4}}\|\nabla
  \um\|_2^{\frac{7}{4}}\|\nabla
  d_t\um\|_2^{\frac{3}{4}}\|d_t\um\|_2^{\frac{1}{4}}
  \\
  &\leq c \alpha^3\kappa (\|u_0\|_2^{2}+\alpha^{2}\|\nabla
  u_0\|_2^{2})^{\frac{1}{8}}\|\nabla
  \um\|_2^{\frac{7}{4}}\|\nabla
  d_t\um\|_2^{\frac{3}{4}}\|d_t\um\|_2^{\frac{1}{4}},
\end{align*}
where in the second line we used~\eqref{eq:a}. By using Young inequality with $p_1=2$,
$p_2=\frac{8}{3}$ and $p_3=8$ and we get
\begin{equation}
  \label{eq:p2}
\begin{aligned}
  \be^3 \kappa|((\um\cdot\nabla)\,\um,d_t\um)|&\leq
  c\kappa(\|u_0\|_2^{2}+\alpha^{2}\|\nabla
  u_0\|_2^{2})^{\frac{1}{4}}\alpha^{\frac{3}{2}}\|\nabla \um\|_2^{\frac{3}{2}}\|\nabla
  \um\|_2^{2}\,
  \\
  &\qquad + \frac{\alpha^3}{2}\kappa\|d_t\um\|_2^2+\frac{\alpha^5}{2}\kappa\|\nabla 
  d_t\um\|_2^2.
\end{aligned}
\end{equation}
Then, by using again~\eqref{eq:a} we have that $\alpha^{\frac{3}{2}}\|\nabla
\um\|_2^{\frac{3}{2}} \leq\big(\|u_0\|_2^{2}+\alpha^{2}\|\nabla
u_0\|_2^{2}\big)^{\frac{3}{4}}$, and then inequality~\eqref{eq:p1} becomes
\begin{equation*}
\begin{aligned}
  \be^3\|\nabla\um\|_2^2-&\be^3\|\nabla\umu\|_2^2+\be^3\|\nabla\um-\nabla\umu\|_2^2 
  \\
  &+\be^3\kappa\|d_t\um\|^2+{\be^5\kappa}\|_2^2d_t\nabla\um\|_2^2 \leq
  c\kappa\|\nabla\um\|^2,
 \end{aligned} 
\end{equation*}
where $c$ is a positive constant depending only on the initial datum $u_0$. By summing up
over $m=1,...,M$ we get~\eqref{eq:tdw}.\par
To prove~\eqref{eq:2dw} we multiply by $-\Delta\um$ the equations~\eqref{eq:a1} and after
integration by parts in space we get
 \begin{equation*}
 \begin{aligned}
 &(d_t\nabla\um,\nabla\um)+\alpha^2(d_t\Delta\um, \Delta\um)
+\|\Delta\um\|^2-(P_n((\um\cdot\nabla)\,\um)\cdot\Delta\um)=0.
 \end{aligned}
 \end{equation*}
 By using~\eqref{eq:elem}, the fact that $\Delta\um\in V_n$, and H\"older inequality we get
 \begin{equation}
   \label{eq:4}
   \begin{aligned}
     \|\nabla\um\|_{2}^2+&\alpha^2\|\Delta\um\|_{2}^2-{\|\nabla\umu\|_{2}^2-
       \alpha^2\|\Delta\umu\|_{2}^2} 
     \\
     &+{\|\nabla\um-\nabla\umu\|_{2}^2}+\be^2{\|\Delta\um-\Delta\umu\|_{2}^2} 
     \\ 
     &\qquad+2\kappa\|\Delta\un\|_2^2
     \leq2\kappa\|\um\|_{4}\|\nabla\um\|_{4}\|\Delta\um\|_{2}.
   \end{aligned}
 \end{equation}
  Then, by Gagliardo Nirenberg inequality and
  Young inequality we have that 
   \begin{equation}\label{eq:5}
  \begin{aligned}
      \kappa\|\um\|_{4}\|\nabla\um\|_{4}\|\Delta\um\|_{2}
      &
      \leq\kappa\|\um\|_{2}^{\frac{1}{4}}\|\nabla\um\|_{2}\|\Delta\um\|_{2}^{\frac{7}{4}}
      \\
      &\leq c\kappa\|\um\|_{2}^2\|\nabla\um\|_{2}^8+\frac{\kappa\|\Delta\um\|_{2}^2}{2}. 
     \end{aligned}
  \end{equation}
 Then, by inserting~\eqref{eq:5} in~\eqref{eq:4} and using~\eqref{eq:a} we get 
 \begin{equation}\label{eq:6}
    \begin{aligned}
      &\|\nabla\um\|_{2}^2+\alpha^2\|\Delta\um\|_{2}^2-\|\nabla\umu\|_{2}^2-
      \alpha^2\|\Delta\umu\|_{2}^2
    \\
    &\qquad+\|\nabla\um-\nabla\umu\|_{2}^2+\be^2\|\Delta\um-\Delta\umu\|_{2}^2
    \\
    &\qquad\qquad+\kappa\|\Delta\un\|_2^2\leq c\kappa\|\um\|_{2}^2\|\nabla\um\|_{2}^8.
  \end{aligned}
\end{equation}
By multiplying the previous inequality on both side by $\be^6$, using again~\eqref{eq:a}
and summing up over $m=1,...,M$ we get~\eqref{eq:2dw} with a constant $c$ independent of
$\alpha$, $n$ and of $M$, thus ending the proof.
\end{proof}
Finally, we prove an {\em a priori} estimate on the approximate pressure, which as usual
is a crucial step when considering the local energy inequality.
\begin{lemma}
  \label{lem:2.3}
  Let $u_0\in H^{2}_{0,\sigma}$. Let $M\in\N$ and $m=1,...,M$.  Let $\um$ be a solution
  of~\eqref{eq:poiss}. Then, there exists $c>0$, independent of
  $\be>0$, of $n\in \N$, and of $M\in\N$  such that 
\begin{equation}
  \label{eq:poiss1}
  \kappa\sum_{m=1}^{M}\|p^{\be,m}_n\|_{{ \frac{5}{3}}}^{\frac{5}{3}}\leq c.
\end{equation}
\end{lemma} 
\begin{proof}
The proof is rather standard. We recall that by Gagliardo-Nirenberg inequality we have 
\begin{equation}
\label{eq:pr1}
\|\um\|_{\frac{10}{3}}\leq \|\um\|_{2}^{\frac{2}{5}}\|\nabla\um\|_{2}^{\frac{3}{5}}.
\end{equation}
By using the $L^{q}$-elliptic estimates applied to~\eqref{eq:poiss} we have that
\begin{equation*}
\|p^{\be,m}_n\|_{\frac{5}{3}}\leq c\|\um\|_{\frac{10}{3}}^2.
\end{equation*}
Then, by using~\eqref{eq:pr1} we have 
\begin{equation}
  \label{eq:pr2}
  \begin{aligned}
    \|p^{\be,m}_n\|_{\frac{5}{3}}^{\frac{5}{3}}&\leq\|\um\|_{2}^{\frac{4}{3}}\|\nabla\um\|_{2}^{2}\leq
    c\|\nabla\um\|_{2}^{2}, 
  \end{aligned}
\end{equation}
where we used~\eqref{eq:a}. By multiplying both sides of~\eqref{eq:pr2} by $\kappa$, by
summing up over $m=1,...,M$, and by using again the equality~\eqref{eq:a} we
get~\eqref{eq:poiss1}.
\end{proof}
At this point we re-state the {\em a priori} estimates proved in  
Lemmas~\ref{lem:discene}-\ref{lem:2.3} in terms of the (time-dependent) functions defined
in~\eqref{eq:vm}.
 \begin{proposition}
  \label{prop:1}
  Let $u_0\in H^2_{0,\sigma}$. There exists $c>0$, independent of $\be>0$, of
  $M\in\N$ and of $n\in\N$, such that
  \begin{align}
    &\|v^{\be,M}_n\|_{L^\infty(L^2)\cap L^2(H^1)}\leq c\label{eq:e1},
    \\
    &\|\partial_t v^{\be,M}_n\|_{L^{4/3}(H^{-2})}\leq c\label{eq:e2},
    \\
    &\|u^{\be,M}_n\|_{L^\infty(L^2)\cap L^2(H^1)}\leq c\label{eq:e3},
    \\
    &\|p^{\be,M}_n\|_{L^{5/3}(L^{5/3})}\leq c\label{eq:e4},
    \\
    \be&\|\nabla v^{\be,M}_n\|_{L^{2}(L^2)}\leq c,\label{eq:e5}
    \\
    \be^{\frac{3}{2}}&\|\partial_t v^{\be,M}_n\|_{L^{2}(L^2)}\leq c,\label{eq:tdwc}
    \\
    \be^{3}&\|\Delta u^{\be,M}_n\|_{L^{2}(L^2)}\leq c.\label{eq:2dwc}
\end{align}
Moreover, we also have the following identities
\begin{align}
  &\|v^{\be,M}_{n}-u^{\be,M}_{n}\|_{L^{2}(0,T;L^2(\tore))}^2
  =\frac{\kappa}{3}\sum_{m=1}^M\|\um-\umu\|_{2}^2, 
  \label{eq:newest1}
  \\
  &\|\nabla{u}^{\be,M}_{n}-\nabla
  v^{\be,M}_{n}\|_{L^{2}(0,T;L^2(\tore))}^2=\frac{\kappa}{3}\sum_{m=1}^M\|\nabla
  \um-\nabla\umu\|_{2}^2.
 \label{eq:newest2}
\end{align}
\end{proposition}
\begin{proof}
  The bound~\eqref{eq:e1} follows from Lemma~\ref{lem:discene} and the
  definition~\eqref{eq:vm}. We remark that in order to get the bound in
  $L^{2}(0,T;H^{1}_{0,\sigma})$ we need $u_0\in H^{1}_{0,\sigma}$. The
  bounds~\eqref{eq:e3},~\eqref{eq:e4} and~\eqref{eq:e5} follow from the definitions
  in~\eqref{eq:vm} and Lemma~\ref{lem:discene}. Finally, the bound~\eqref{eq:e2} follows
  by a simple comparison argument on~\eqref{eq:a7}. The bounds~\eqref{eq:tdwc}
  and~\eqref{eq:2dwc} follows by Lemma~\ref{lem:2} and~\eqref{eq:vm} and the
  identities~\eqref{eq:newest1} and~\eqref{eq:newest2} follow by a direct calculation.
\end{proof}
%
\section{Proof of the main Theorem}
\label{sec:5}
In this section we give the proof of Theorem~\ref{teo:main}. We divide the proof in two
main steps: a) the convergence to a Leray-Hopf weak solution and b) the convergence
to a suitable weak solution. Let $\{M_n\}_{n}\subset\N$ and $\{\be_n\}_{n}\subset(0,1)$
be two sequences as in the statement of Theorem~\ref{teo:main}. We recall that
$\{\be_n\}_{n}$ is chosen such that
\begin{equation}
  \label{eq:hypo}
  \lim_{n\to+\infty}n\be_n^3=0.
\end{equation}
\noindent {\em Step 1: Convergence to a Leray-Hopf weak solution}\par
Let $\varphi\in C^{\infty}_{c}([0,T);C^{\infty}(\tore))$ with $\nabla\cdot\varphi=0$ and
zero mean value. It is easy to show that there exists a sequence
$\{\varphi_n\}_{n}\subset C^{1}([0,T);V_n)$ such that
\begin{equation}
  \label{eq:c}
  \sup_{t\in(0,T)}\|\varphi_n-\varphi\|_{H^1}+\|\partial_{t}(\varphi_n-\varphi)\|_{H^1}\to
  0,\textrm{ as }n\to\infty. 
\end{equation}
In order to simplify the exposition we use the following abbreviations: 
\begin{equation*}
  \vn:=v^{\be_n,M_n}_n,\qquad\un:=u^{\be_n,M_n}_n,\quad\text{and}\quad
  \pn:=p^{\be_n,M_n}_n. 
\end{equation*}
Then,~\eqref{eq:a7} reads as follows
\begin{equation}
  \label{eq:c1}
\begin{aligned}
  &\partial_t(\vn-\alpha^2\Delta\vn)-\Delta\un+(\un\cdot\nabla)\,\un
  -Q_{n}((\un\cdot\nabla)\,\un)+\nabla \pn=0.
\end{aligned}
\end{equation}
We recall from~\eqref{eq:e1} and~\eqref{eq:e2} that (with bounds independent of $n$)
\begin{equation*}
\begin{aligned}
  &\vn\in L^{\infty}(0,T;L_{0,\sigma}^2)\cap L^{2}(0,T;H_{0,\sigma}^{1}),
  \\
  &\partial_t\vn\subset L^{\frac{4}{3}}(0,T;H^{-2}). 
\end{aligned}
\end{equation*}
Then, there exists $v\in L^{\infty}(0,T;L_{0,\sigma}^2)\cap L^{2}(0,T;H_{0,\sigma}^{1})$ such that, 
up to a subsequence not relabelled, 
\begin{equation*}
\vn\to v\textrm{ strongly in }L^{2}(0,T;L_{0,\sigma}^2),\textrm{ as }n\to\infty.
\end{equation*}
Next, from~\eqref{eq:e3} there exists $u\in L^{\infty}(0,T;L_{0,\sigma}^2)\cap
L^{2}(0,T;H_{0,\sigma}^{1})$ such that, up to a subsequence not relabelled,
\begin{equation}
  \label{eq:c4}
\un\weakto u\textrm{ weakly in }L^{2}(0,T;H_{0,\sigma}^1),\textrm{ as }n\to\infty.
\end{equation}
Finally, by using~\eqref{eq:newest1} we have 
\begin{equation*}
\begin{aligned}
  \int_0^T\|\un-\vn\|_{2}^{2}\,dt&=\frac{T}{3M_n}\sum_{m=1}^{M_n}\|\um-\umu\|_{2}^{2}
\leq \frac{T}{3M_n} (\|u_0\|_{2}^2+\be_n^2\|\nabla u_0\|_{2}^{2}),
\end{aligned}
\end{equation*}
where we used Lemma~\ref{lem:discene}. We have then that 
\begin{equation}
  \label{eq:c5}
  \un-\vn\to 0\textrm{ strongly in }L^{2}(0,T;L_{0,\sigma}^2),\textrm{ as }n\to\infty.
\end{equation}
Hence, it follows that $u=v$ and also that 
\begin{equation}
  \label{eq:c6}
\begin{aligned}
  &\un\to u\textrm{ strongly in }L^{2}(0,T;L_{0,\sigma}^2),\textrm{ as }n\to\infty,
  \\
  &\vn\to u\textrm{ strongly in }L^{2}(0,T;L_{0,\sigma}^2),\textrm{ as }n\to\infty.
\end{aligned}
\end{equation}
Let $\varphi_n$ satisfying~\eqref{eq:c}, by multiplying~\eqref{eq:c1} by $\varphi_n$ and
by integrating by parts with respect to space and time we get
\begin{equation*}
    \int_0^T(\vn,\partial_t\varphi_n)-\be_n^2(\nabla\vn, \partial_t\nabla\varphi_n)
    +(\un\otimes\un,\nabla\varphi_n)-(\nabla\un,\nabla\varphi_n)=(P_nu_0,\varphi_0(0)).  
\end{equation*}
By using~\eqref{eq:e5}, we have  then 
\begin{equation*}
  \be_n^2\int_{0}^{T}\|\nabla\vn\|^{2}\,dt\leq c. 
\end{equation*}
This implies, in particular, that
\begin{equation*}
  \be_n^2\int_{0}^{T}(\Delta\partial_t\vn,\varphi_n)\,dt\to 0,\textrm{ as }n\to\infty. 
\end{equation*}
Then, by using~\eqref{eq:c},~\eqref{eq:c4}, and~\eqref{eq:c6} it is now straightforward to
prove the convergence to a Leray-Hopf weak solution. We omit further details.
\medskip

\noindent {\em Step 2: Convergence to a Suitable Weak Solution}\par

We prove now the most original part of this work, namely that the limit of the approximate
solutions satisfy the local energy inequality. By using~\eqref{eq:e4} we can infer there
exists $p\in L^{\frac{5}{3}}((0,T)\times\tore)$ such that (again up to a subsequence)
\begin{equation}
  \label{eq:c8}
  \pn\weakto p\textrm{ weakly in }L^{\frac{5}{3}}((0,T)\times\tore),\textrm{ as
  }n\to\infty.  
\end{equation}
In order to prove that $(u,p)$ is a suitable weak solution we only need to prove that
$(u,p)$ satisfies the local energy inequality~\eqref{eq:lei}. To this end we consider the
equations~\eqref{eq:a8} that we rewrite for the reader's convenience
\begin{equation}
  \label{eq:c9}
  \partial_t\vn-\alpha^2\partial_t\Delta\vn-\Delta\un+(\un\cdot\nabla)\,\un-
  Q_{n}((\un\cdot\nabla)\,\un)+\nabla\pn=0. 
\end{equation}
By testing~\eqref{eq:c9} by $\un\phi$ with $\phi\in
C_c^{\infty}((0,T);C^{\infty}(\tore))$, $\phi\geq 0$, and after several integration by
parts we get 
\begin{equation}
  \label{eq:c10}
  \begin{aligned}
    \int_{0}^T\int_{\tore}|\nabla\un|^2\phi\,dxdt=&-\int_0^T(\partial_t\vn,\un\phi)\,dt
    +\be_n^2\int_0^T(\partial_t\Delta\vn,\un\phi)\,dt
    \\
    &+\int_0^T\left(\frac{|\un|^2}{2},\Delta\phi\right)\,dt
    +\int_0^T\int_{\tore}\left(\frac{|\un|^2}{2}+p\right)\un\cdot\nabla\phi\,dx
    dt
    \\
    &+\int_0^T(Q_n(\un\cdot\nabla)\,\un),\un\phi)\,dt=:\sum_{i=1}^5I_{i}^{n}.
  \end{aligned}
\end{equation}
We treat all the terms on the right-hand side of~\eqref{eq:c10} separately. We start by
$I_1^n$.  
\begin{equation*} 
\begin{aligned}
  I_{1}^{n}&=-\int_0^T(\partial_t\vn,\un\phi)\,dt
=-\int_{0}^T(\partial_t\vn,\vn\phi)+\int_{0}^{T}(\partial_t\vn,(\vn-\un)\phi)\,dt
  \\
  &=\int_0^T\left(\frac{|\vn|^2}{2},\partial_t\phi\right)+
  \sum_{m=1}^{M_{n}}\int_{t_{m-1}}^{t_m}(\partial_t\vn, (\vn-\un)\phi)\,dt
\end{aligned}
\end{equation*}
By using that $\un$ is constant over $[t_{m-1},t_m)$, we infer that 
\begin{equation*}
  \begin{aligned}
\sum_{m=1}^{M_n}\int_{t_{m-1}}^{t_m}(\partial_t\vn, (\vn-\un)\phi)\,dt
    &=\sum_{m=1}^{M_n}\int_{t_{m-1}}^{t_m}(\partial_t(\vn-\un), (\vn-\un)\phi)\,dt
    \\
    &=-\sum_{m=1}^{M_n}\int_{t_{m-1}}^{t_m}\left(\frac{|\vn-\un|^2}{2},\partial_t\phi\right)\,dt,
  \end{aligned}
\end{equation*}
and we point out that there are no boundary terms arising in integration by parts due to
the fact that $\vn(t_m)=\un(t_m)$ for any $m=1,...,M_n$ and $\phi$ is compactly supported
in time.  Then,
\begin{equation*}
  I^{n}_{1}=\int_{0}^{T}\left(\frac{|\vn|^2}{2}-\frac{|\vn-\un|^2}{2},\partial_t\phi\right)\,dt,
\end{equation*}
and by using~\eqref{eq:c6} and~\eqref{eq:c5} it follows 
\begin{equation}
  \label{eq:c11}
  I^{n}_1\to\int_{0}^{T}\left(\frac{|u|^{2}}{2},\partial_t\phi\right)\,dt,\textrm{ as
  }n\to\infty.
\end{equation}
Let us consider now the term $I^{n}_{2}$. We have
\begin{equation*}
\begin{aligned}
  I_{2}^{n}=\be_2^2\int_0^T(\partial_t\Delta\vn,\un\phi)\,dt
  &=\be_n^2\int_{0}^T(\partial_t\Delta\vn,(\un-\vn)\phi)\,dt+
  \be_n^2\int_{0}^{T}(\partial_t\Delta\vn,\vn\phi)\,dt  
  \\
  &=:I^{n}_{2,1}+I^{n}_{2,2}. 
\end{aligned}
\end{equation*}
We estimate the term $I^{n}_{2,1}$ in a way similar to the term $I^{n}_{1,2}$. By using
that $\un$ is constant  over the interval $[t_{m-1},t_{m})$ we get
\begin{equation*}
\begin{aligned}
  I^{n}_{2,1}&=-\be_n^2\sum_{m=1}^{M_n}\int_{t_{m-1}}^{t_m}(\partial_t\nabla\vn,\nabla
  (\vn-\un)\phi)\,dt
  \\
  &=\be_n^2\sum_{m=1}^{M_n}\int_{t_{m-1}}^{t_m}(\partial_t\nabla(\vn-\un),
  \nabla(\vn-\un)\phi)\,dt
  \\
  &=-\be_n^2\sum_{m=1}^{M_n}\int_{t_{m-1}}^{t_m}
  \left(\frac{|\nabla(\vn-\un)|^2}{2},\partial_t\phi\right)\,dt 
  \\
  &=-\be_n^2\int_0^T\left(\frac{|\nabla(\vn-\un)|^2}{2},\partial_t\phi\right)\,dt,
\end{aligned}
\end{equation*}
where we used that $\nabla\vn(t_m)=\nabla\un(t_{m})$ for any $m=1,...,M_n$ and again that
$\phi$ is compactly supported in time. By using~\eqref{eq:newest2} we have (for a constant
$c$ depending only on $\phi$) 
\begin{equation*}
\begin{aligned}
  |I^{n}_{2,1}|&\leq c\,\be_n^2\int_0^T\|\nabla\vn-\nabla\un\|_{2}^2
  \\
  &=\frac{c\,T}{3M_n}\be_n^2\sum_{m=1}^{M_n}\|\nabla\um-\nabla\umu\|_{2}^{2}
  \\
  &\leq\frac{c\,T}{3M_n}(\|u_0\|_{2}^{2}+\be_n^2\|\nabla u_0\|_{2}^{2})\to 0,\textrm{ as
  }n\to\infty. 
\end{aligned}
\end{equation*}           
Now we consider the term $I_{2,2}^n$. By standard manipulations involving integrations by
parts we get that
\begin{equation*}
  \begin{aligned}
    I^{n}_{2,2}&= \be_n^2\int_0^T\int_{\tore}\Delta \partial_t\vn \vn \phi\,dx dt
    \\
    &=\be_n^2\int_0^T\int_{\tore}\left[\frac{|\nabla \vn|^2}{2}\partial_t\phi+ \nabla \vn
    \nabla\phi\,\partial_t\vn-\frac{|\vn|^2}{2} \Delta\,\partial_t\phi\right]\,dx dt
    \\
    & \leq\frac{\be_n^2}{2}\int_0^T\int_{\tore}|\nabla\vn|^2|\partial_t\,\phi|\,dx
    dt+\frac{\be_n^2}{2}\int_0^T\int_{\tore}|\vn|^2|\Delta\,\partial_t\phi|\,dx dt
    \\
    &\quad+\be_n^2\int_0^T\int_{\tore}|\partial_t\vn|\,|\nabla\vn|\,|\nabla\phi|\,dx dt
    \\
    & \leq c\be_n^2+c\be_n^2\int_0^T\|\partial_t\vn\|_2\|\nabla\vn\|_2\,dt
  \end{aligned}
\end{equation*}
where we used~\eqref{eq:e1}, H\"older inequality, and the fact that $\phi\in
C_c^\infty((0,T)\times\tore)$. Then,
 \begin{equation*}
 \begin{aligned}
   |I^{n}_{2,2}|&\leq c\be_n^2+c\be_n^2\int_0^T\|\partial_t\vn\|_2\|\nabla\vn\|_2\,dt
   \\
   &\leq
   c\be_n^2+c\be_n^{\frac{1}{2}}\left(\int_0^T\be_n^3\|\partial_t\vn\|_2^2\right)^{\frac{1}{2}}
   \left(\int_0^T\|\nabla\vn\|_2^2\right)^{\frac{1}{2}}
   \\
   &\leq c\,(\be_n^2+\be_n^{\frac{1}{2}})\to 0,\textrm{ as }n\to\infty.
 \end{aligned}
\end{equation*}
where we used H\"older inequality in time and~\eqref{eq:tdwc}. In particular, we have just
proved that
\begin{equation}
  \label{eq:c13}
  |I^{n}_{2}|\leq |I^{n}_{2,1}|+|I^{n}_{2,2}|\to 0,\textrm{ as }n\to\infty.
\end{equation}
Concerning the term $I^{n}_{3}$ and $I^{n}_{4}$ we recall that from~\eqref{eq:c4}
and~\eqref{eq:c6}
\begin{equation}
  \label{eq:c14}
  \un\to u\textrm{ strongly in }L^{3}(0,T;L^{3}(\tore)),\textrm{ as }n\to\infty. 
\end{equation}
Then,~\eqref{eq:c14} and~\eqref{eq:c8} are enough to prove that 
\begin{align}
  I^{n}_{3}&=\int_{0}^{T}\left(\frac{|\un|^2}{2},\Delta\phi\right)\,dt\ \to\
  \int_{0}^{T}\left(\frac{|u|^2}{2},\Delta\phi\right)\,dt ,\textrm{ as
  }n\to\infty,\label{eq:c15}
  \\
  I^{n}_{4}&=\int_{0}^{T}\left(\left(\frac{|\un|^2}{2}+\pn\right)\un,\nabla\phi\right)\,dt\
  \to\ \int_{0}^{T}\left(\left(\frac{|u|^2}{2}+p\right)u,\nabla\phi\right)\,dt, \textrm{
    as }n\to\infty.\label{eq:c16}
\end{align}
We are left with the term $I^{n}_{5}$. We have
\begin{equation*}
\begin{aligned}
  I^{n}_{5}&=\int_{0}^{T}(Q_n((\un\cdot\nabla)\,\un,\un\phi)\,dt
  =\int_{0}^{T}((\un\cdot\nabla)\,\un,Q_n(\un\phi))\,dt
  \\
  &\leq \int_0^T\|\un(t)\|_2\|\nabla\un(t)\|_2\|Q_n(\un(t)\phi(t))\|_{\infty}\,dt
  \\
  &\leq c\left(\int_0^T\|Q_n(\un(t)\phi(t))\|_{\infty}^2\,dt\right)^{\frac{1}{2}},
\end{aligned}
\end{equation*}
where in the last line we used H\"older inequality and~\eqref{eq:e3}. Then,
from~\eqref{eq:a4} and~\eqref{eq:vm} we have that $\un$ has the following representation
in Fourier series expansion
\begin{equation*}
\un(t,x)=\sum_{0<|k|\leq
    n}\sum_{m=1}^{M_n}\chi_{[t_{m-1},t_m)}(t)\hat{u}_{n,k}^{\,\be_n,m}\,\text{e}^{ik\cdot x}. 
\end{equation*}
By defining 
\begin{equation*}
\widehat{U}^{n}_{k}(t):=\sum_{m=1}^{M_n}\chi_{[t_{m-1},t_m)}(t)\hat{u}_{n,k}^{\,\be_n,m},
\end{equation*}
we have that 
\begin{equation*}
u^n(t,x)=\sum_{0<|k|\leq
  n}\widehat{U}_k^{n}(t)\,\text{e}^{ik\cdot x}.
\end{equation*}
Then, by using Lemma~\ref{lem:2.4} we have that
\begin{equation*}
\begin{aligned}
\int_0^T\|Q_n(\un(t)\phi(t))\|_{\infty}^2\,dt&\leq \frac{c}{n}
      \sum_{k\in\mathbb{Z}^3}|{\widehat{U}}_{n}^{k}(t)|^2\,dt+c\int_0^T
      n^2\sum_{|k|\geq\frac{n}{2}}|{\widehat{U}}_{n}^{k}(t)|^2
      =:I^{n}_{5,1}+I^{n}_{5,2}.
\end{aligned}
\end{equation*}
Regarding the term $I^{n}_{5,1}$ it follows by~\eqref{eq:e3} that 
\begin{equation*}
|I^{n}_{5,1}|\leq \frac{c}{n}\to 0,\textrm{ as }n\to\infty. 
\end{equation*}
For the term $I^{n}_{5,2}$ we have 
\begin{equation*}
\begin{aligned}
  \int_0^T n^2\sum_{|k|\geq\frac{n}{2}}|\widehat{U}^{n}_{k}(t)|^2&=
  \frac{n^2\be_n^6}{n^2\be_n^6}\int_0^T\sum_{|k|\geq\frac{n}{2}}
  n^2|\widehat{U}^{n}_{k}(t)|^2\,dt
  \\
  &\leq 4
  \frac{\be_n^6}{n^2\be_n^6}\int_0^T\sum_{|k|\geq\frac{n}{2}}|k|^4|\widehat{U}^{n}_{k}(t)|^2\,dt
  \\
  &\leq \frac{4}{n^2\be_n^6}\
  \be_n^6\int_0^T\sum_{k\in\Z^3\backslash\{0\}}|k|^4|\widehat{U}^{n}_{k}(t)|^2\,dt
  \\
  &\leq \frac{c}{n^2\be_n^6} \ \be_n^6\int_0^T\|\Delta
  \un\|^2_2\,dt\leq \frac{c}{n^2\be_n^6} .
\end{aligned}
\end{equation*}
where in the last inequality we have used~\eqref{eq:2dwc}. 
Then, by~\eqref{eq:hypo} we get that $|I^{n}_{5,2}|\to 0$ as $n\to\infty$ and then
\begin{equation}
  \label{eq:c17}
  |I^{n}_{5}|\to 0,\textrm{ as }n\to\infty.
\end{equation}
Finally, by using~\eqref{eq:c4} we have that 
\begin{equation}
  \label{eq:c18}
  \int_{0}^{T}\int|\nabla u|^2\phi\,dxdt\leq\liminf_{n\to\infty}\int_{0}^{T}\int|\nabla
  \un|^2\phi\,dxdt. 
\end{equation}
By
inserting~\eqref{eq:c18},~\eqref{eq:c11},~\eqref{eq:c13},~\eqref{eq:c15},~\eqref{eq:c16},
and~\eqref{eq:c17} in~\eqref{eq:c10} we have finally proved the local energy
inequality~\eqref{eq:lei}.

\section*{Acknowledgement}
S. Spirito acknowledges the support by INdAM-GNAMPA.

\end{document}